\documentclass[psamsfonts]{amsart}

\usepackage{amssymb,amsfonts}
\usepackage[all,arc]{xy}
\usepackage{enumerate}
\usepackage{mathrsfs}
\usepackage{url}

\newtheorem{thm}{Theorem}[section]
\newtheorem{cor}[thm]{Corollary}
\newtheorem{prop}[thm]{Proposition}
\newtheorem{lem}[thm]{Lemma}

\newtheorem{quest}[thm]{Question}

\theoremstyle{definition}
\newtheorem{defn}[thm]{Definition}

\newtheorem{fact}[thm]{Fact}

\theoremstyle{remark}
\newtheorem{rem}[thm]{Remark}

\newcommand{\CL}{{\mathcal L}} 
\newcommand{\la}{\langle} 
\newcommand{\ra}{\rangle} 
\newcommand{\CT}{{\mathcal T}} 
\newcommand{\dom}{\mbox{dom}}

\makeatletter
\let\c@equation\c@thm
\makeatother
\numberwithin{equation}{section}

\def\Ind{\setbox0=\hbox{$x$}\kern\wd0\hbox to 0pt{\hss$\mid$\hss} \lower.9\ht0\hbox to 0pt{\hss$\smile$\hss}\kern\wd0} 

\def\Notind{\setbox0=\hbox{$x$}\kern\wd0\hbox to 0pt{\mathchardef \nn=12854\hss$\nn$\kern1.4\wd0\hss}\hbox to 0pt{\hss$\mid$\hss}\lower.9\ht0 \hbox to 0pt{\hss$\smile$\hss}\kern\wd0} 

\def\ind{\mathop{\mathpalette\Ind{}}} 

\def\nind{\mathop{\mathpalette\Notind{}}} 

\def\tp{\operatorname{tp}}

\def\Lstp{\operatorname{Lstp}}

\title{Independence over arbitrary sets in NSOP$_{1}$ theories}

\date{\today}

\begin{document}

\author[J. Dobrowolski]{Jan Dobrowolski}
\address{School of Mathematics, University of Leeds, LS2 9JT, Leeds, UK}
\email{J.Dobrowolski@leeds.ac.uk}

\author[B. Kim]{Byunghan Kim}
\address{Department of Mathematics, Yonsei University \\
50 Yonsei-ro Seodaemun-gu \\
Seoul 03722 \\
Republic of Korea}
\email{bkim@yonsei.ac.kr}

\author[N. Ramsey]{Nicholas Ramsey}
\address{Department of Mathematics, UCLA\\
Los Angeles\\
USA}
\email{nickramsey@math.ucla.edu}

\thanks{The first author was supported  by European Union's Horizon 2020 research
and innovation  programme under the Marie Sklodowska-Curie grant agreement No 705410, by NCN Grant no. 2015/19/D/ST1/01174, and by the Foundation for Polish Science (FNP).
The second author has been supported by  Samsung Science Technology Foundation under Project Number SSTF-BA1301-03 and 
 an NRF of Korea grant 2018R1D1A1A02085584. }

\maketitle

\begin{abstract}
We study Kim-independence over arbitrary sets.  Assuming that forking satisfies existence, we establish Kim's lemma for Kim-dividing over arbitrary sets in an NSOP$_{1}$ theory.  We deduce symmetry of Kim-independence and the independence theorem for Lascar strong types.  
\end{abstract}

\setcounter{tocdepth}{1}
\tableofcontents

\section{Introduction}

Kim-independence is a notion of independence that unifies and explains simplicity-like phenomena in a non-simple setting.  Non-forking independence, introduced originally by Shelah, allowed for the organization of core properties of simple theories into a collection of basic principles\textemdash the `non-forking calculus'\textemdash that made possible the deepening of simplicity theory and the streamlined treatment of examples with simple theory.  Following these developments, increased attention was given to classes of theories where a notion of independence might serve a similar explanatory role.  This was motivated from below by the study of independence and amalgamation in natural examples\textemdash the generic parametrized equivalence relations of D\v{z}amonja and Shelah, Granger's generic vector spaces with a bilinear form, and Frobenius fields studied by Chatzidakis\textemdash as well as from above, by the desire to situate the theorems of simplicity theory, especially those dealing with independence relations, in their broadest possible setting.  It was known from the work of Chernikov and the third-named author \cite{ArtemNick} that the existence of a well-behaved notion of independence implies that a theory is NSOP$_{1}$ and the theory of Kim-independence allows one to reverse the implication, using the hypothesis of NSOP$_1$ alone to develop a notion of independence that specializes to natural notions of independence in the examples mentioned above. Since it was introduced in \cite{kaplan2017kim}, it has been established that almost all of the important properties of non-forking independence in simple theories, with the exception of base monotonicity, carry over to Kim-independence in the NSOP$_{1}$ setting.

However, one major limitation for the theory of Kim-independence is that the theory, up to this point, has only been developed for the special situation where one considers independence \emph{over a model}.  The reason for this is that Kim-independence is defined to be independence \emph{at a generic scale}\textemdash it is defined in terms of Kim-dividing formulas, that is formulas that divide with respect to a generic sequence. The initial definition of Kim-dividing took `generic sequence' to mean Morley sequence in a global invariant type.  In any theory, given any type over a model, one can construct a Morley sequence in a global finitely satisfiable, hence invariant, type extending it, which means that the idea of dividing along a generic sequence is not vacuous.  Over arbitrary sets, it is certainly the case that one cannot expect every type to have a finitely satisfiable or even invariant global extension, but one might instead consider those theories in which every complete type has a global \emph{non-forking} extension\textemdash a (non-forking) Morley sequence in this type would furnish a notion of generic sequence allowing one to define Kim-independence over arbitrary sets.   This property, that every type extends to a global non-forking type, is equivalent to what we call \emph{non-forking existence}, which asserts that no type forks over its domain.  As far as we know, \emph{all} NSOP$_{1}$ theories that have been studied satisfy non-forking existence, and it is known that every NSOP$_{1}$ theory is interpretable in an NSOP$_{1}$ theory in which this holds \cite{kruckman2018generic}.      

Here we show that the entire theory of Kim-independence can be extended to give a notion of independence over arbitrary sets in NSOP$_{1}$ theories with non-forking existence.  We modify the definition of Kim-dividing so that a formula $\varphi(x;a)$ Kim-divides over $A$ if it divides along some non-forking Morley sequence over $A$. In the context of an NSOP$_{1}$ theory, this gives an equivalent notion of independence when $A$ is a model but this equivalence over models is obtained as a corollary after the theory of Kim-independence has been developed, using invariant types at every stage.  Consequently, new methods were required in extending the theory to arbitrary sets.  This is most pronounced in the proof of Kim's lemma, which reduces Kim's lemma over \emph{sets} to Kim's lemma over \emph{models}, instead of merely adapting the existing proof.  This uses methods from classical stability, related to the fundamental order, as well as the construction of suitable trees.    

As a consequence of the theory of independence developed here, we prove independent 3-amalgamation for Lascar types in NSOP$_{1}$ theories with non-forking existence.  In simple theories, the independence theorem  over a model was proved in \cite{KimPillay}, from which 3-amalgamation for Lascar types was deduced as a corollary.  Later, Shami proved 3-amalgamation for Lascar types in a simple theory directly \cite{Shami}, which has the result `over models' as an immediate consequence.  Shami's proof influences the approach to 3-amalgamation for $\ind^{K}$ taken here, but we still proceed by deducing the theorem from the independence theorem over models for $\ind^{K}$ in NSOP$_{1}$ theories.  Indeed, our proof relies on Kim's lemma, extension, and the symmetry of Kim-dividing over arbitrary sets in NSOP$_{1}$ theories with existence, and the independence theorem over models for $\ind^{K}$ is an essential ingredient in the proofs of these properties.  It seems unlikely that there is a direct proof of 3-amalgamation over sets without appealing to the corresponding result over models in NSOP$_{1}$ theories.  Lastly, we mention that any simple theory is NSOP$_{1}$ with non-forking existence and in a simple theory, Kim-independence and non-forking independence coincide.  Our results, therefore, generalize these theorems from simplicity theory.  

In the final section, we introduce a syntactic property called the \emph{skew tree property}, or STP, which is equivalent to NSOP$_{1}$ among theories satisfying non-forking existence and may be characterized by a certain amalgamation property.  The general question of whether STP is equivalent to SOP$_{1}$ bears a resemblance to the question of whether SOP$_{1}$ and SOP$_{2}$ are equivalent (both are questions about whether one can `upgrade' a tree assumed to have a limited quantity of inconsistency to one with considerably more) and offers a test case for whether a theory of independence might be useful for collapsing or separating syntactic dividing lines.  We end with several open problems.

\section{The existence axiom}

We let $T$ be a complete theory in a language $\mathcal{L}$.  We fix a monster model $\mathbb{M} \models T$ which is $\overline{\kappa}$-saturated and $\overline{\kappa}$-homogeneous for some sufficiently large $\overline{\kappa}$\textemdash we will implicitly assume all models and parameter sets are of size $<\overline{\kappa}$ and contained in $\mathbb{M}$.  

\subsection{Forking and dividing}
We recall the following standard definitions:
\begin{defn}
Let $A$ be a set of parameters.  
\begin{enumerate}
\item Let $k$ be a natural number. We say a formula
$\varphi(x,a_0)$ {\em $k$-divides} over $A$ if for some $A$-indiscernible sequence $\langle a_i : i<\omega\rangle$ in $\mathrm{tp}(a_0/A)$,
$\{ \varphi(x, a_i)|\ i<\omega\}$ is $k$-inconsistent (meaning its each $k$-element subset is inconsistent).
\item A formula divides over $A$ if it $k$-divides over $A$ for some $k$.
\item A formula $\varphi(x;a)$ \emph{forks} over $A$ if $\varphi(x;a) \vdash \bigvee_{i < k} \psi_{i}(x;b_{i})$ where $\psi_{i}(x;b_{i})$ divides over $A$ for each $i < k$.  
\item A type $p(x)$ forks or divides over $A$ if it implies a formula that forks or divides over $A$, respectively.  
\item We write $a \ind_{A} b$ and $a \ind^{d}_{A} b$ to denote the assertions that $\text{tp}(a/Ab)$ does not fork over $A$ and does not divide over $A$, respectively .  
%
%
\end{enumerate}
\end{defn}
We will make free use of the following facts about the relation $\ind$ (see, e.g. \cite{kim2013simplicity}):
\begin{fact}  
Let $A$ be a set of parameters.
\begin{enumerate}
\item Extension:  For any $c$, if $a \ind_{A} b$, there is $a' \equiv_{Ab} a$ so that $a' \ind_{A} bc$.  
\item Left-transitivity:  If $a \ind_{A} c$ and $b \ind_{Aa} c$ then $ab \ind_{A} c$.  
\item Base monotonicity:  If $a \ind_{A} bc$ then $a \ind_{Ab} c$.
\end{enumerate}
(2) and (3) hold for $\ind^{d}$ as well.
\end{fact}

\begin{defn}\cite[Definition 2.2.3]{kim2013simplicity}
Let $A\subseteq B$ and $p \in S(B)$. By a Morley sequence in $p$ over $A$, we mean a
$B$-indiscernible sequence $\langle a_i : i<\omega\rangle$ of realizations of $p$ such that for every $i<\omega$,
$\tp(a_i/Ba_{<i})$ does not fork over $A$. When $A = B$, we omit mentioning “over $A$.”
\end{defn}

In order to construct Morley sequences, we will often use the following consequence of the Erd\H{o}s-Rado theorem: 
\begin{fact} \cite{shelah1980simple} \label{er}
 Let $B$ be a set of parameters and $\kappa$ a cardinal. For any sequence $(a_i)_{i<\beth((2^{|T|+|B|+\kappa})^+)}$ consisting
of tuples of length $\leq \kappa$ there is a $B$-indiscernible 
sequence $(a_j')_{j<\omega}$ based on $(a_i)_{i<\beth(2^{|T|+|B|+\kappa})^+}$ over $B$, i.e., : For every 
$k<\omega$ there are $i_0<i_1<\dots< i_{k-1}$ such that $a_0'a_1'\dots a_{k-1}'\equiv_B a_{i_0}a_{i_1}\dots a_{i_{k-1}}.$
\end{fact}

\begin{fact}\label{seqamal}
\begin{enumerate}
\item \cite[Lemma 2.2.5]{kim2013simplicity}
Let $I=\langle a_i\rangle_i$ be a Morley 
 sequence in $p=\mathrm{tp}(a_i/A)$, and let $J=\langle b_j\rangle_j$ be an arbitrary $A$-indiscernible in $p$.
 Then there is $I'\equiv_A I$ such that 
 $b_jI'$  is an $A$-indiscernible for each $b_j$.
 \item Let $I^{\frown}J$ be a Morley 
 sequence over $A$. Then due to left-transitivity and base monotonicity of nonforking independence, we have $J\ind_A I$. 
\end{enumerate}
\end{fact}

Recall that  the {\em existence axiom} (with respect to nonforking independence) says  that  any complete type over a set does not fork over the set, and that
a \emph{global type} is a type in $S(\mathbb{M})$. 

\begin{rem}
The following are equivalent.
\begin{enumerate}
\item $T$ satisfies the existence axiom.
\item For all parameter sets $A$, no consistent formula over $A$ forks over $A$.  
\item For all parameter sets $A$, every type $p \in S(A)$ has a global extension that does not fork over $A$.
\item For all parameter sets $A$ and any $p\in S(A)$, there is a Morley sequence in $p$.
\end{enumerate}
\end{rem}

\begin{proof}
(1)$\implies$(2) is clear.

(2)$\implies$(3):  given a type $p \in S(A)$, consider the partial type
$$
p(x) \cup \{\neg \varphi(x;c) : \varphi(x;c) \text{ forks over }A\}.
$$ 
By compactness, if this partial type is inconsistent, there is some $\psi(x) \in p$ such that 
$$
\psi(x) \vdash \bigvee_{i < k} \varphi_{i}(x;c_{i})
$$
where each $\varphi_{i}(x;c_{i})$ forks over $A$, which means that $\psi(x)$ is a formula over $A$ that forks over $A$, contradicting (2).  Therefore this partial type is consistent, and any completion gives the desired global extension.

(3)$\implies$(4)  Given $p$, let $q$ be a global extension that does not fork over $A$.  Let $\kappa$ be sufficiently large and choose the sequence $\langle a_{i} : i < \kappa \rangle$ so that $a_{i} \models q|_{Aa_{<i}}$.  Then $a_{i} \ind_{A} a_{<i}$ for all $i < \kappa$.  Applying Lemma \ref{er}, we obtain the desired Morley sequence.

(4)$\implies$(1):  If $p$ is an arbitrary complete type over $A$, let $\langle a_{i} : i < \omega \rangle$ be a Morley sequence in $p$.  Then in particular, $a_{1} \models p$ and $a_{1} \ind_{A} a_{0}$, which implies $a_{1} \ind_{A} Aa_{0}$ and hence $a_{1} \ind_{A} A$.  This shows that the existence axiom is satisfied.  
\end{proof}

A global type $q\in S\left(\mathbb{M}\right)$ is called $A$\emph{-invariant} if $b\equiv_{A}b'$ implies $\varphi\left(x;b\right)\in q$ if and
only if $\varphi\left(x;b'\right)\in q$. A global type $q$ is \emph{invariant} if there is some small set $A$ such that $q$ is $A$-invariant.We write $a\ind_{M}^{u}B$ to mean $\text{tp}\left(a/MB\right)$ is finitely satisfiable in $M$, in other words $\text{tp}(a/MB)$ is a \emph{coheir} of its restriction to $M$.  We say $\text{tp}(a/MB)$ is an \emph{heir} of its restriction to $M$ if $B \ind^{u}_{M} a$.  

\begin{defn}
Suppose $q$ is an $A$-invariant global type and $I$ is a linearly
ordered set. By a \emph{Morley sequence in }$q$ \emph{over} $A$
\emph{of order type} $I$, we mean a sequence $\langle b_{\alpha}: \alpha\in I\rangle$
such that for each $\alpha\in I$, $b_{\alpha}\models q|_{Ab_{<\alpha}}$
where $b_{<\alpha}=\langle b_{\beta} : \beta<\alpha \rangle$. Given a linear
order $I$, we will write $q^{\otimes I}$ for the global $A$-invariant
type in variables $\langle x_{\alpha}: \alpha \in I \rangle$ such that for any
$B\supseteq A$, if $\overline{b}\models q^{\otimes I}|_{B}$ then
$b_{\alpha}\models q|_{Bb_{<\alpha}}$ for all $\alpha\in I$. If
$q$ is, moreover, finitely satisfiable in $A$, in which case $b_{\alpha} \ind^{u}_{A} b_{<\alpha}$ for all $\alpha \in I$, then we refer to
a Morley sequence in $q$ over $A$ as a \emph{coheir sequence} over
$A$.  Likewise, an $A$-indiscernible sequence $\langle a_{i} : i < \omega \rangle$ such that $a_{<i} \ind^{u}_{A} a_{i}$ is called an \emph{heir sequence} over $A$.  
\end{defn}

\subsection{NSOP$_{1}$ and Kim-independence}

\begin{defn} \cite[Definition 2.2]{dvzamonja2004maximality} \label{sop1def}
The formula $\varphi(x;y)$ has SOP$_{1}$ if there is a collection of tuples $(a_{\eta})_{\eta \in 2^{<\omega}}$ so that 
\begin{itemize}
\item For all $\eta \in 2^{\omega}$, $\{\varphi(x;a_{\eta | \alpha}) : \alpha < \omega\}$ is consistent.
\item For all $\eta \in 2^{<\omega}$, if $\nu \unrhd \eta \frown \langle 0 \rangle$, then $\{\varphi(x;a_{\nu}), \varphi(x;a_{\eta \frown 1})\}$ is inconsistent.
\end{itemize}
We say $T$ is SOP$_{1}$ if some formula has SOP$_{1}$ modulo $T$.  $T$ is NSOP$_{1}$ otherwise.  
\end{defn}

\begin{fact}\label{arrayequivalent}  \cite[Lemma 5.1]{ArtemNick} \cite[Proposition 2.4]{kaplan2017kim}
The following are equivalent:
\begin{enumerate}
\item $T$ has SOP$_{1}$.  
\item There is a formula $\varphi(x;y)$ and an array $(c_{i,j})_{i < \omega, j < 2}$ so that 
\begin{enumerate}
\item $c_{i,0} \equiv_{\overline{c}_{<i}} c_{i,1}$ for all $i < \omega$.
\item $\{\varphi(x;c_{i,0}) : i < \omega\}$ is consistent.
\item $\{\varphi(x;c_{i,1}) :i < \omega\}$ is $2$-inconsistent.
\end{enumerate}
\item There is a formula $\varphi(x;y)$ and an array $(c_{i,j})_{i < \omega, j < 2}$ so that 
\begin{enumerate}
\item $c_{i,0} \equiv_{\overline{c}_{<i}} c_{i,1}$ for all $i < \omega$.
\item $\{\varphi(x;c_{i,0}) : i < \omega\}$ is consistent.
\item $\{\varphi(x;c_{i,1}) :i < \omega\}$ is $k$-inconsistent for some $k$.
\end{enumerate}
\end{enumerate}
\end{fact}

\begin{defn}
Suppose $A$ is a set of parameters.  
\begin{enumerate}
\item We say a formula
$\varphi(x,a_0)$ {\em Kim-divides} over $A$ if for some Morley sequence $\langle a_i : i<\omega\rangle$ in $\mathrm{tp}(a_0/A)$,
$\{ \varphi(x, a_i)|\ i<\omega\}$ is inconsistent. 
\item A formula $\varphi(x;a)$ {\em Kim-forks} over $A$ if $\varphi(x;a) \vdash \bigvee_{i < k} \psi_{i}(x;b_{i})$ where $\psi_{i}(x;b_{i})$ Kim-divides over $A$ for all $i < k$.  
\item Likewise we say a type $p(x)$ Kim-forks or Kim-divides over $A$ if it implies a formula that Kim-forks or Kim-divides over $A$, respectively.  
\item We write $a \ind^{K}_{A} b$ to denote the assertion that $\text{tp}(a/Ab)$ does not Kim-fork over $A$.  
%
%
\end{enumerate}
\end{defn}

In \cite{kaplan2017kim}, the following definition of Kim-dividing was introduced:  $\varphi(x;a)$ Kim-divides over $A$ if $\{\varphi(x;a_{i}): i < \omega\}$
 is inconsistent for some  Morley sequence 
$\langle a_i\mid i<\omega\rangle$ in an $A$-invariant global type with $a=a_0$.  However, by \cite[Theorem 7.7]{kaplan2017kim}, the definition of Kim-dividing given above is equivalent to the definition of Kim-dividing over models in NSOP$_{1}$ theories.  There are known examples even of simple theories where types over sets (even assuming they are boundedly closed) do not have global invariant extensions 
(see \cite{fff}) so the above definition will be the more fruitful one for extending the theory to arbitrary sets.  

We remark that if $A$ and $B$ are sets of parameters and $p$ is a partial type over $B$ that does not Kim-fork over $A$, then there is a complete type $q$ over $B$ extending $p$ which also does not Kim-fork over $A$.  This does not require existence and follows from the same argument as the analogous statement for non-forking:  one simply takes any completion of $p$ together with the negation of all formulas over $B$ which Kim-fork over $A$, which is consistent by compactness.  

The following facts summarize the key properties of Kim-independence that have been established over models:

\begin{fact} \label{basic kimindep facts}
\cite[Theorem 3.15]{kaplan2017kim}  \label{kimslemma} The following
are equivalent for the complete theory $T$: 
\begin{enumerate}
\item $T$ is NSOP$_{1}$. 
\item Kim's lemma for Kim-dividing over models: Given any model $M\models T$ and formula
$\varphi\left(x;b_{0}\right)$, $\varphi\left(x;b_{0}\right)$ Kim-divides over $M$ if and only if for any $\langle b_{i} : i < \omega \rangle$ Morley over $M$ in some global $M$-invariant type, $\{\varphi\left(x;b_{i}\right) : i < \omega\}$ is inconsistent.
\item Symmetry of Kim independence over models: $a\ind_{M}^{K}b$ iff $b\ind_{M}^{K}a$
for any $M\models T$.
\item Independence theorem over models: if $A\ind^{K}_{M}B$, $c\ind^{K}_{M}A$,
$c'\ind^{K}_{M}B$ and $c\equiv_{M}c'$ then there is some $c''\ind^{K}_{M}AB$
such that $c''\equiv_{MA}c$ and $c''\equiv_{MB} c'$.
\item If $\varphi(x;a_0)$ does not Kim-divide over $M$ and $\langle a_{i} : i < \omega \rangle$ is an $M$-indiscernible sequence with $a_{i} \ind_{M}^{K} a_{<i}$ then $\{\varphi(x;a_{i}) : i < \omega\}$ is consistent.  
\end{enumerate}
\end{fact}

\begin{lem}\label{uniformitylemma}
Assume $T$ satisfies the existence axiom.  Suppose that $I=\langle a_i : i< \omega\rangle$ is Morley over $A$. Then for any $b'$ there are $b_i$ such that 
$a_0b'\equiv_A a_0b_0 $ and $\langle a_ib_i|\ i<\omega\rangle$ is Morley over $A$.
\end{lem}
\begin{proof}
We may stretch $I$ to $\langle a_{i} : i < \kappa \rangle$ for a sufficiently large $\kappa$.  By induction on $i < \kappa$, we will choose $a'_{i}b'_{i}$ so that 
\begin{enumerate}
\item $a'_{i}b'_{i} \equiv_{A} a_{0}b'$ and $a'_{\leq i} \equiv_{A} a_{\leq i}$.    
\item $a'_{i} \ind_{A} a'_{<i}b'_{<i}$.
\item $b'_{i} \ind_{Aa'_{i}} a'_{<i} b'_{<i}$ for all $i < \kappa$.  
\end{enumerate}
Suppose for some $i < \kappa$, we have chosen $a'_{<i}b'_{<i}$ satisfying the above requirements.  Since $a'_{<i} \equiv_{A} a_{<i}$, there is $a'_{i}$ such that $a'_{<i} a'_{i} \equiv_{A} a_{<i} a_{i}$ and therefore $a'_{i} \ind_{A} a'_{<i}$.  By extension, we may assume $a'_{i} \ind_{A} a'_{<i} b'_{<i}$.  By existence and extension, there is $b'_{i}$ such that $a'_{i} b'_{i} \equiv_{A} a_{0} b'$ and $b'_{i} \ind_{Aa'_{i}} a'_{<i}b'_{<i}$.  The pair $a'_{i}b'_{i}$ satisfies the requirements and, by induction, completes the construction of $\langle a'_{i} b'_{i} : i < \kappa \rangle$.

Note that, by the left-transitivity of non-forking independence, conditions (2) and (3) imply that $a'_{i}b'_{i} \ind_{A} a'_{<i}b'_{<i}$ for all $i < \kappa$.  By applying Erd\H{o}s-Rado (Fact \ref{er}), we obtain an $A$-indiscernible sequence $\langle a''_{i} b''_{i} : i < \kappa \rangle$ which also satisfies (1)-(3).  By (1) and an automorphism, we can find $b_{i}$ for all $i < \kappa$ so that $\langle a''_{i} b''_{i} : i < \kappa \rangle \equiv_{A} \langle a_{i} b_{i} : i < \kappa \rangle$, which gives the desired sequence.

\end{proof}

Finally, we observe that the property of Kim-dividing depends only on the set defined by a formula, not on the formula itself.  

\begin{cor}\label{uniformity}
Assume $T$ satisfies the existence axiom.  
\begin{enumerate}
\item Kim-dividing of a formula depends only on the definable set defined by the formula.
\item If $\varphi_0(x,a)\models \varphi_1(x,b)$ and  $\varphi_1(x,b)$ Kim-divides over $A$, then so does $\varphi_0(x,a)$.
\end{enumerate}
\end{cor}
\begin{proof}
 (1) The assertion is clear for Kim-forking, so we prove it only for Kim-dividing. Let $\models \varphi(x,a_0)\leftrightarrow\psi(x,b_0)$ and assume $\{ \varphi(x,a_i): i < \omega\}$ is inconsistent for some Morley sequence $I=\langle a_i : i < \omega \rangle$ over $A$. 
 By Lemma \ref{uniformitylemma}, there is a Morley sequence  $\langle b_i : i < \omega \rangle$ over $A$ such that $\varphi(x,a'_i)\leftrightarrow\psi(x,b_i)$ with some $\langle a'_i\rangle\equiv_AI$, so $\{ \psi(x,b_i)\}_i$ is inconsistent.  This shows that if $\varphi(x;a_{0})$ Kim-divides over $A$, then $\psi(x;b_{0})$ Kim-divides over $A$, and by symmetry we conclude.  (2) follows from Lemma \ref{uniformitylemma} as well, by an entirely similar argument.  
\end{proof}
Clearly, the above corollary remains true if we replace Kim-dividing by Kim-forking, even if we do not assume the existence axiom.

Finally, we note that there are many NSOP$_{1}$ theories for which our results apply.  The existence axiom has been proved explicitly in almost every NSOP$_{1}$ theory that has been studied in detail, and there is no known example of an NSOP$_{1}$ theory in which existence fails.  

\begin{fact} \label{examples}
The following theories are NSOP$_{1}$ and satisfy the existence axiom:
\begin{enumerate}
\item Any complete theory of $\omega$-free PAC fields. \cite[Proposition 3.1]{chatzidakis2002properties} \cite[Corollary 6.2]{ArtemNick}
\item The theory $T_{m,n}$ of existentially closed incidence structures omitting the complete
incidence structure $K_{m,n}$.  \cite[Theorem 4.11, Corollary 4.24]{conant2019independence}
\item The theory $T^{*}_{feq}$ of parameterized equivalence relations.  \cite[Corollary 6.4]{ArtemNick}
\item For any language $L$, the model completion of the empty theory $T^{\emptyset}_{L}$. \cite[Corollary 3.13, Proposition 3.17]{kruckman2018generic}
\item The theory ACFG of an algebraically closed field of characteristic $p$ with a predicate naming a generic multiplicative subgroup. \cite[Theorem 5.9]{d2018generic} \cite[Corollary 3.13]{d2018forking}
\end{enumerate}
\end{fact}

\begin{rem}
Although it does not appear explicitly in the literature, we sketch how existence may be shown for $T^{*}_{feq}$.  Recall that the language contains two unary relations $P$ and $O$, and a ternary relation $E_{x}(y,z)$.  The theory $T_{feq}$ asserts that $P$ and $O$ are disjoint, $E_{x}(y,z)$ implies $x \in P$ and $(y,z) \in O^{2}$, and, for every $p \in P$, $E_{p}(y,z)$ is an equivalence relation on $O$. Given $p \in P$ and $c \in O$, we write $[c]_{E_{p}}$ for the $E_{p}$-class of $c$.  The model completion $T^{*}_{feq}$ is $\aleph_{0}$-categorical, with trivial algebraic closure and elimination of quantifiers, see \cite[Subsection 6.3]{ArtemNick} for further details.

Suppose $\mathbb{M}$ is a monster model of $T^{*}_{feq}$, $A \subseteq \mathbb{M}$ is a small set of parameters, and $q \in S(A)$.  We can write $q = q(x_{0},\ldots, x_{n-1},y_{0},\ldots, y_{m-1})$ where $q \vdash x_{i} \in O$ and $q \vdash y_{j} \in P$ for all $i < n,j < m$.  We, moreover, may easily reduce to the case that $q$ implies no equality between an $x_{i}$ or $y_{j}$ and an element of $A$.  Define a global type extending $q$ with the following formulas:
\begin{itemize}
\item For all $i < n$, $p \in P(\mathbb{M}) \setminus A$, $c \in O(\mathbb{M})$, we add $\neg E_{p}(x_{i},c)$.
\item For all $i < n$, $p \in P(A)$, and $c \in \mathbb{M}$ such that $[c]_{E_{p}} \cap A = \emptyset$, we put $\neg E_{p}(x_{i},c)$.  
\item For all $j < m$ and $c \in \mathbb{M}$, we add $y_{j} \neq c$.
\item For all $j < m$, $c \in \mathbb{M}$, and $c' \in \mathbb{M} \setminus A$, we add $\neg E_{y_{j}}(c,c')$, unless $c = c'$.
\end{itemize}
It is easy to check that this implies a complete global type and is moreover $A$-invariant, hence does not fork over $A$.  As $q$ and $A$ are arbitrary, this shows existence for $T^{*}_{feq}$.
\end{rem}

\subsection{Transfer to models}

In this subsection, we explain how, from a Morley sequence over $A$, one can find a model $M$ containing $A$ such that the sequence is finitely satisfiable in $M$ and $M$
 is independent from the sequence over $A$.  This will be a key step in reducing Kim's lemma for Kim-dividing over arbitrary sets 
to the known version of Kim's lemma for Kim-dividing over models.  Our proof makes use of notions related to the fundamental order from classical stability theory (see \cite{Lascar79}).  
\textbf{Throughout this subsection we assume the existence axiom.}

For $A\subseteq M$ and finite $a$, put
$$\mathrm{Cl}_A(a/M):=\{ \varphi(x,y)\in \mathcal{L}(A)|\  \varphi(x,m)\in \mathrm{tp}(a/M)\mbox{ for some }m\in M\}.$$

\begin{lem}\label{heirbase} Let $I=\langle a_i|\ i<\omega\rangle$ be an $A$-Morley sequence. Then there is a model $M\supseteq A$ such that 
$M\ind_A I$,  $I$ is
 $M$-indiscernible, and  $\mathrm{tp}(a_{< k}/Ma_{\geq k})$ is an heir extension of $\mathrm{tp}(a_{< k}/M)$ for any $  k<\omega$ (so $I$ is a coheir sequence over $M$). 
\end{lem}
\begin{proof} 
Consider the following class of models
$$\mathcal{U}_0:=\{N \prec \mathbb{M} | \ A\subseteq N, \mbox{ }|N| < \overline{\kappa}, \mbox{ $I$ is $N$-indiscernible, and $N\ind_A I$}\},$$
which is ordered as follows: $N_1< N_2$ if $N_1\prec N_2$ and $\mathrm{Cl}_A(I/N_1)$ (with object variables $x_0,x_1,\ldots$ corresponding to $a_0,a_1,\ldots$) is a proper subset of  $\mathrm{Cl}_A(I/N_2).$
Due to the existence axiom, there is a model $M'\supseteq A$ such that $M'\ind_A I$. Then, by Ramsey's theorem and compactness, there is $I'\equiv_A I$
 such that $I'$ is $M'$-indiscernible and $M'\ind_A I'$:  by Ramsey and compactness, from $I$, we extract an $M'$-indiscernible sequence $I' = \langle a'_{i} : i < \omega \rangle$.  Since $I$ was already $A$-indiscernible, we have $I' \equiv_{A} I$.  If $M' \nind_{A} I'$, then there are $i_{0} < \ldots < i_{n-1}$ and $\varphi(x;a'_{i_{0}},\ldots, a'_{i_{n-1}}) \in \text{tp}(M'/AI')$ which forks over $A$.  But then, as $I'$ was extracted from $I$, there are $j_{0} < \ldots < j_{n-1}$ so that $\varphi(x;a_{j_{0}},\ldots, a_{j_{n-1}}) \in \text{tp}(M'/AI)$.  Since $a'_{i_{0}}\ldots a'_{i_{n-1}} \equiv_{A} a_{j_{0}}\ldots a_{j_{n-1}}$, it follows that $\varphi(x;a_{j_{0}},\ldots, a_{j_{n-1}})$ also forks over $A$, contradicting $M' \ind_{A} I$.  This shows $M' \ind_{A} I'$.  
 
Hence by an $A$-automorphism sending $I'$ to $I$, we see that $\mathcal{U}_0$ is non-empty.  Moreover, any chain in $\mathcal{U}_{0}$ has length at most $|\mathcal{L}(A)|$ so its union is also in $\mathcal{U}_{0}$.  Therefore, by Zorn's Lemma there is a maximal element in $\mathcal{U}_0$, say $M_0$. 

\medskip

\noindent\textbf{Claim.} For any $k<l<\omega$, an $\mathcal{L}(A)$-formula $\varphi(\bar x, y,\bar z)$ and  $m\in M_0$ such that
 $\models \varphi(a_{<k},m,a_k\dots a_l)$, 
there is  $m'\bar m''\in M_0$ such that $\varphi (a_{<k},m',\bar m'')$ holds.

\medskip

\noindent\emph{Proof of Claim.}  By compactness, clearly there is
 $I'$ such that $II'$ is $M_0$-indiscernible and $A$-Morley, and $M_0\ind_AII'$. We can further  assume the length of $I$ is sufficiently large. Note that $M_0\ind_{AI'}I $, and,  by Fact \ref{seqamal}, $I'\ind_AI$. Thus, by left-transitivity, $M_0I'\ind_A I$ holds. Now, again by existence, 
there is  $M'_0\supseteq M_0I'$ such that
$M'_0\ind_{M_0I'}I$. Then since $I$ is $M_0I'$-indiscernible, we can, as above, apply Ramsey and compactness to assume $I$ (with its original length $\omega$) is $M'_0$-indiscernible.  Due to left-transitivity we have $M'_0\ind_A I$, so by the maximality of $M_0$, we have $\mathrm{Cl}_A(I/M_0)=\mathrm{Cl}_A(I/M'_0)$. Moreover, there are $a'_k\ldots a'_l\in I'$ such that 
$a_{<k}a_k\ldots a_l\equiv_{M_0} a_{<k}a'_k\ldots a'_l$. Then  $\models \varphi(a_{<k},m,a'_k\dots a'_l)$, so 
$\varphi(\bar x; y,\bar z)\in \mathrm{Cl}_A(I/M'_0)=\mathrm{Cl}_A(I/M_0)$, which gives the claim. \qed

\medskip

As we considered formulas over $A$ (not over $M_0$), $M_0$ is not yet a desired model. Now we iterate this argument to 
obtain an elementary chain
 $\{M_i|\ i<\omega\}$ such that $M_{n+1}$ is a maximal element in 
 $$\mathcal{U}_{n+1}:=\{N| \ M_n\prec N, \mbox{ and $I$ is $N$-indiscernible, and $N\ind_A I$}\},$$
ordered by: $N_1< N_2$ if $N_1\prec N_2$ and $\mathrm{Cl}_{M_n}(I/N_1)\subsetneq \mathrm{Cl}_{M_n}(I/N_2).$  Then by the same argument as above, the same Claim (except 
$\varphi$ is  assumed to be an $\CL(M_n)$-formula, $m\in M_{n+1}$ and $m'\bar m''\in M_{n+1}$)  holds. Therefore we put 
$M:=\bigcup_{n<\omega} M_n$, which is a desired model.
\end{proof}
 
By  choosing a maximal element in $\mathcal{U}_0$ in  the proof of Lemma \ref{heirbase} suitably,  indeed we have proved the following.  
 
\begin{lem}\label{heirbase2} Let $I$ be an $A$-Morley sequence.  Then there is a model $M \supseteq A$ such that $M \ind_{A} I$ and $I$ is a coheir sequence over $M$.  Moreover, if we are given $M_0\supseteq A $ such that $M_0\ind_A I$ and  $I$ is $M_0$-indiscernible, then there is a model $M\succ M_0$ such that  $M\ind_A I$ and  $I$ is a coheir sequence over $M$.
\end{lem}

\section{Kim's lemma over arbitrary sets}\label{section_Klemma}
{\bf Until the end of Section \ref{sec_symmetry} we assume $T$ has NSOP$_1$ and satisfies existence for forking independence}.

For an ordinal $\alpha$, let the language $L_{s,\alpha}$ be $\langle \unlhd, \wedge, <_{lex}, (P_{\beta})_{\beta \leq \alpha} \rangle$.  We may view a tree with $\alpha$ levels as an $L_{s,\alpha}$-structure by interpreting $\unlhd$ as the tree partial order, $\wedge$ as the binary meet function, $<_{lex}$ as the lexicographic order, and $P_{\beta}$ interpreted to define level $\beta$. 

Recall the modeling property:
\begin{defn}  Suppose $I$ is an $L'$-structure, where $L'$ is some language. 
\begin{enumerate}
\item  We say $(a_{i} : i \in I)$ is a set of $I$\emph{-indexed indiscernibles over $A$} if whenever 

$(s_{0}, \ldots, s_{n-1})$, $(t_{0}, \ldots, t_{n-1})$ are tuples from $I$ with 
$$
\text{qftp}_{L'}(s_{0}, \ldots, s_{n-1}) = \text{qftp}_{L'}(t_{0}, \ldots, t_{n-1}),
$$
then we have
$$
\text{tp}(a_{s_{0}},\ldots, a_{s_{n-1}}/A) = \text{tp}(a_{t_{0}},\ldots, a_{t_{n-1}}/A).
$$
\item In the case that $L' = L_{s,\alpha}$ for some $\alpha$, we say that an $I$-indexed indiscernible is $\emph{s-indiscernible}$.  As the only $L_{s,\alpha}$-structures we will consider will be trees, we will often refer $I$-indexed indiscernibles in this case as \emph{s-indiscernible trees}.  
\item We say that $I$-indexed indiscernibles have the \emph{modeling property} if, given any $(a_{i} : i \in I)$ from $\mathbb{M}$ and any $A$, there is an \(I\)-indexed indiscernible \((b_{i} : i \in I)\) over $A$ in $\mathbb{M}$ \emph{locally based} on \((a_{i} : i \in I)$ over $A$ -- i.e., given any finite set of formulas \(\Delta\) from $\CL(A)$ and a finite tuple \((t_{0}, \ldots, t_{n-1})\) from \(I\), there is a tuple \((s_{0}, \ldots, s_{n-1})\) from \(I\) so that 
\[
\text{qftp}_{L'} (t_{0}, \ldots, t_{n-1}) =\text{qftp}_{L'}(s_{0}, \ldots , s_{n-1})
\]
and also 
\[
\text{tp}_{\Delta}(b_{t_{0}}, \ldots, b_{t_{n-1}}) = \text{tp}_{\Delta}(a_{s_{0}}, \ldots, a_{s_{n-1}}).
\]
\end{enumerate}
\end{defn}

\begin{fact}\cite[Theorem 4.3]{KimKimScow}\label{modeling}
Let  \(I_{s}\) denote the \(L_{s,\omega}\)-structure \((\omega^{<\omega}, \unlhd, <_{lex}, \wedge, (P_{\alpha})_{\alpha < \omega})\) with all symbols being given their intended interpretations and each \(P_{\alpha}\) naming the elements of the tree at level \(\alpha\).  Then \(I_{s}\)-indexed indiscernibles have the modeling property.  
\end{fact}

\begin{rem}
It follows by compactness that for any cardinal $\kappa$, if $(a_{\eta})_{\eta \in \omega^{<\kappa}}$ is a collection of tuples, then there is an $s$-indiscernible tree $(b_{\eta})_{\eta \in \omega^{<\kappa}}$ locally based on $(a_{\eta})_{\eta \in \omega^{<\kappa}}$.  For arbitrary $\kappa$, one considers the partial type $\Gamma(x_{\eta} : \eta \in \omega^{<\kappa})$ consisting of formulas naturally asserting that $(x_{\eta})_{\eta \in \omega^{<\kappa}}$ is $s$-indiscernible, together with every formula of the form $\varphi(\overline{x}_{\overline{\eta}})$ where $\models \varphi(\overline{a}_{\overline{\nu}})$ for all tuples $\overline{\nu}$ from $\omega^{<\kappa}$ realizing $\mathrm{qftp}_{L_{s,\kappa}}(\overline{\eta})$.  Fact \ref{modeling} may be used to show any finite subset is satisfiable and a realization will be the desired $s$-indiscernible tree. 
\end{rem}

\begin{lem} \label{building tree}
Suppose $A$ is a set of parameters, $I = \langle a_{i} : i < \omega \rangle$ and $J = \langle b_{i} : i < \omega\rangle$ are $A$-indiscernible sequences with $a_{0} = b_{0}$ and $b_{>0} \ind^{d}_{A} b_{0}$.  Then there is a tree $(c_{\eta})_{\eta \in \omega^{<\omega}}$ satisfying the following properties:
\begin{enumerate}
\item For all $\eta \in \omega^{<\omega}$, $(c_{\eta \frown \langle i \rangle})_{i < \omega} \equiv_{A} I$.  
\item For all $\eta \in \omega^{<\omega}$, $(c_{\eta},c_{\eta | l(\eta) - 1}, \ldots, c_{\emptyset}) \equiv_{A} (b_{0},b_{1},\ldots, b_{l(\eta)})$.  
\item $(c_{\eta})_{\eta \in \omega^{<\omega}}$ is $s$-indiscernible over $A$. 
\end{enumerate}
\end{lem}

\begin{proof}
By induction on $n$, we will construct a tree $(c_{\eta})_{\eta \in \omega^{\leq n}}$ so that 
\begin{enumerate}
\item For all $\eta \in \omega^{<n}$, $(c_{\eta \frown \langle i \rangle})_{i < \omega} \equiv_{A} I$.
\item For all $\eta \in \omega^{n}$, $(c_{\eta}, c_{\eta | (n-1)}, \ldots, c_{\eta |0}) \equiv_{A} (b_{0},\ldots, b_{n})$.  
\end{enumerate}
For $n = 0$, we may set $c_{\emptyset} = b_{0}$, which trivially satisfies the requirements.  Now suppose we are given $(c_{\eta})_{\eta \in \omega^{\leq n}}$.  By the indiscernibility of $J$ and (2), we have $(c_{0^{n}},c_{0^{n-1}},\ldots, c_{\emptyset}) \equiv_{A} (b_{1},\ldots, b_{n+1})$ so we may choose $c_{*}$ so that $c_{*}(c_{0^{n}},c_{0^{n-1}},\ldots, c_{\emptyset}) \equiv_{A} b_{0}(b_{1},\ldots, b_{n+1})$.  By invariance and $b_{>0} \ind^{d}_{A} b_{0}$, we have $c_{0^{n}}c_{0^{n-1}}\ldots c_{\emptyset} \ind^{d}_{A} c_{*}$.  Let $I_{*} = \langle c_{*,i} : i < \omega\rangle$ be a sequence with $I_{*} \equiv_{A} I$ and $c_{*,0} = c_{*}$.  Since $c_{0^{n}}c_{0^{n-1}}\ldots c_{\emptyset} \ind^{d}_{A} c_{*}$, we may assume $I_{*}$ is $Ac_{0^{n}}c_{0^{n-1}}\ldots c_{\emptyset}$-indiscernible and, therefore, that for all $i < \omega$, 
$$
c_{*,i} c_{0^{n}}\ldots c_{\emptyset} \equiv_{A} b_{0} b_{1} \ldots b_{n+1}.  
$$

Now we define a tree $(c_{\eta})_{\eta \in \omega^{\leq n+1}}$ by placing a copy of $I$ on top of each node of level $n$ in $(c_{\eta})_{\eta \in \omega^{\leq n}}$.  More precisely, using (2), we may choose, for each $\eta \in \omega^{n}$, an automorphism $\sigma_{\eta} \in \text{Aut}(\mathbb{M}/A)$ so that $\sigma_{\eta}(c_{0^{n}} c_{0^{n-1}} \ldots c_{\emptyset}) = c_{\eta} c_{\eta |(n-1)} \ldots c_{\emptyset}$.  We may take $\sigma_{0^{n}} = \text{id}_{\mathbb{M}}$.  Now to define $(c_{\eta})_{\eta \in \omega^{\leq n+1}}$, we put, for each $\eta \in \omega^{n}$ and $i < \omega$, $c_{\eta \frown \langle i \rangle} = \sigma_{\eta}(c_{*,i})$.  By induction and the construction, the tree constructed this way satisfies (1).  Moreover, for all $\eta \in \omega^{n}$, 
$$
c_{\eta \frown \langle i \rangle} c_{\eta} \ldots c_{\emptyset} \equiv_{A} c_{*,i} c_{0^{n}} \ldots c_{\emptyset} \equiv_{A} b_{0}b_{1}\ldots b_{n+1},
$$
so this tree satisfies (2) as well.  

This completes the inductive construction of $(c_{\eta})_{\eta \in \omega^{<\omega}}$.  By applying Fact \ref{modeling} to extract $(c'_{\eta})_{\eta \in \omega^{<\omega}}$ $s$-indiscernible over $A$ and locally based on $(c_{\eta})_{\eta \in \omega^{<\omega}}$, we obtain the desired tree.  
\end{proof}

For the following proof, recall that we say $(a_{i,j})_{i < \kappa, j < \lambda}$ is a \emph{mutually indiscernible array} over $A$ if, for each $i < \kappa$, the sequence $\overline{a}_{i} = \langle a_{i,j} : j < \lambda \rangle$ is indiscernible over $A\overline{a}_{\neq i}$.  

\begin{thm} \label{KLforanyset}
$T$ satisfies Kim's lemma for Kim-dividing over an arbitrary set $A$: a formula $\varphi(x,a)$ Kim-divides over $A$ with respect to some Morley sequence in $\tp(a/A)$ iff it Kim-divides over $A$ with respect to any such sequence. 
\end{thm}

\begin{proof}
Towards contradiction, assume we are given $\varphi(x;a)$ and Morley sequences $I = \langle a_{i} : i < \omega \rangle$ and $J = \langle b_{i} : i < \omega \rangle$ over $A$ both of which are in $\text{tp}(a/A)$ and such that $\{\varphi(x;a_{i}) : i < \omega \}$ is consistent and $\{\varphi(x;b_{i}) : i < \omega\}$ is inconsistent.  

As $J$ is a Morley sequence over $A$, we have $b_{>0} \ind_{A} b_{0}$ so, by Lemma \ref{building tree}, we can find a tree $(c_{\eta})_{\eta \in \omega^{<\omega}}$ which satisfies:
\begin{enumerate}
\item For all $\eta \in \omega^{<\omega}$, $(c_{\eta \frown \langle i \rangle})_{i < \omega} \equiv_{A} I$.  
\item For all $\eta \in \omega^{<\omega}$, $(c_{\eta},c_{\eta | l(\eta) - 1}, \ldots, c_{\emptyset}) \equiv_{A} (b_{0},b_{1},\ldots, b_{l(\eta)})$.  
\item $(c_{\eta})_{\eta \in \omega^{<\omega}}$ is $s$-indiscernible over $A$. 
\end{enumerate}

Define an array $(d_{i,j})_{i,j < \omega}$ by $d_{ij} = c_{0^{i} \frown (j+1)}$. By $s$-indiscernibility, $(d_{i,j})_{i,j < \omega}$ is mutually indiscernible and, moreover, $\overline{d}_{i} \equiv_{A} I$ for all $i < \omega$.  By compactness applied to $(d_{i,j})_{i,j < \omega}$ and $(c_{0^{i+1}})_{i < \omega}$, for $\kappa$ large enough, we can find an $A$-mutually indiscernible array $(e_{i,j})_{i < \kappa, j < \omega}$ and an $A$-indiscernible sequence $(e'_{i})_{i < \kappa}$ having the same EM-type over $A$ as $J$ (with the reversed order) such that $\overline{e}_{i} \equiv_{A} I$ and $e'_{i} \equiv_{Ae_{<i}e'_{<i}} e_{i,0}$ for all $i < \kappa$.  By Fact \ref{er}, we may assume $\kappa = \omega$ and $(\overline{e}_{i})_{i < \omega}$ is $A$-indiscernible.  

By applying Lemma \ref{heirbase2}, we can find a model $M \supseteq A$ such that $\overline{e}_{0}$ is a Morley sequence over $M$ and $M \ind_{A} \overline{e}_{0}$.  Then because $(\overline{e}_{i})_{i < \omega}$ is $A$-indiscernible, we may, moreover, assume $M$ has been chosen so that $(\overline{e}_{i})_{i < \omega}$ is $M$-indiscernible.  Let $\lambda$ be any cardinal larger than $2^{|\mathcal{L}| + |M|}$ and apply compactness to stretch the array to $(e_{i,j})_{i < \omega, j < \lambda}$, preserving $A$-mutual indiscernibility and the $M$-indiscernibility of $(\overline{e}_{i})_{i < \omega}$.  

Now by induction, we will find $\alpha_{n} < \lambda$ so that $e_{n,\alpha_{n}} \ind^{K}_{M} e_{<n, \alpha_{<n}}$ for all $n < \omega$.  Suppose we have succeeded in finding $(\alpha_{m})_{m < n}$.  Then by the pigeonhole principle and the choice of $\lambda$, there is an infinite subsequence $I_{n}$ of $\overline{e}_{n}$ so that every tuple of $I_{n}$ has the same type over $Me_{<n,\alpha_{<n}}$.  Let $\alpha_{n} < \lambda$ be least such that $e_{n,\alpha_{n}} \in I_{n}$.  As $I_{n}$ is a subsequence of a Morley sequence over $M$, $I_{n}$ is also Morley over $M$ and hence, by Kim's lemma for Kim-dividing, we have $e_{<n,\alpha_{<n}} \ind^{K}_{M} e_{n,\alpha_{n}}$ which, by symmetry, is what we need to complete the induction.

We claim that $\{\varphi(x;e_{n,\alpha_{n}}) : n < \omega\}$ is consistent.  By compactness, it suffices to show that $\{\varphi(x;e_{n,\alpha_{n}}) : n < N\}$ does not Kim-divide over $M$ for any $N$.  This is true for $N = 1$ by Kim's lemma, since $\{\varphi(x;e_{0,j}) : j < \lambda\}$ is consistent and $\overline{e}_{0}$ is a Morley sequence over $M$.  Assuming we have shown it for $N$, we can choose $c_{N} \models \{\varphi(x;e_{n,\alpha_{n}}) : n < N\}$ with $c_{N} \ind^{K}_{M} e_{<N,\alpha_{<N}}$.  Additionally, since $e_{0,\alpha_{0}} \equiv_{M} e_{N,\alpha_{N}}$, we can choose $c$ so that $c_{N}e_{0,\alpha_{0}} \equiv_{M} c e_{N,\alpha_{N}}$, from which it follows $c \ind^{K}_{M} e_{N,\alpha_{N}}$ by invariance.  Applying the independence theorem over $M$, we find $c_{N+1} \models \text{tp}(c_{N}/Me_{<N,\alpha_{<N}}) \cup \text{tp}(c/Me_{N,\alpha_{N}})$.  In particular, we have $c_{N+1} \models \{\varphi(x;e_{n,\alpha_{n}}) : n < N+1\}$, and therefore $\{\varphi(x;e_{n,\alpha_{n}}) : n < N+1\}$ does not Kim-divide, completing the induction.

By mutual indiscernibility over $A$, we also have $\{\varphi(x;e_{i,0}) : i < \omega\}$ is consistent.  By NSOP$_{1}$ and Fact \ref{arrayequivalent}, it follows that $\{\varphi(x;e'_{i}) : i < \omega\}$ is also consistent.  But this entails that $\{\varphi(x;b_{i}) : i < \omega\}$ is consistent, a contradiction.  
\end{proof}

\begin{rem}An alternative conclusion to the argument may be given, less elementary but with fewer moving parts.  Starting from the definition of the array $(d_{i,j})_{i,j < \omega}$ by $d_{ij} = c_{0^{i} \frown (j+1)}$, we let $(\overline{d}'_{i})_{i < \omega}$ be an $A$-indiscernible sequence locally based on $(\overline{d}_{i})_{i < \omega}$.  By Lemma \ref{heirbase2}, there is a model $M \supseteq A$ so that $M \ind^{f}_{A} \overline{d}'_{0}$ and $\overline{d}'_{0}$ is a coheir sequence over $M$, and, as above, we may assume $(d'_{i})_{i < \omega}$ is $M$-indiscernible.  Finally, let $(e_{i,j})_{i,j < \omega}$ be an array that is mutually indiscernible over $M$ (see, e.g., \cite[Lemma 1.2]{ChernikovNTP2})and locally based on $(d'_{i,j})_{i,j < \omega}$.  Notice that for all $f : \omega \to \omega$, because $(d_{i,j})_{i,j < \omega}$ and hence $(d'_{i,j})_{i,j<\omega}$ are mutually indiscernible over $A$, we have 
$$
(e_{i,f(i)})_{i < \omega} \equiv_{A} (d'_{i,f(i)})_{i < \omega} \equiv_{A} (d'_{i,0})_{i < \omega}.
$$
Note that for all $i$, $\{\varphi(x;e_{i,j}) : j < \omega\}$ is consistent because $\{\varphi(x;a_{j}) : j < \omega\}$ is consistent and $\overline{e}_{i}$ is a coheir sequence over $M$, so $\varphi(x;e_{i,0})$ does not Kim-divide over $M$.  Moreover, by mutual indiscernibility, $\overline{e}_{i}$ is $M\overline{e}_{<i}$-indiscernible so, in particular, $e_{i,0} \ind^{K}_{M} e_{<i,0}$.  This shows $\{\varphi(x;e_{i,0}) : i < \omega\}$ is consistent by Fact \ref{basic kimindep facts}(5).  Because the indiscernible sequence $(d'_{i,0})_{i < \omega}$ is locally based on $(d_{i,0})_{i < \omega} = (c_{0^{i} \frown 1})_{i < \omega}$, $\{\varphi(x;e_{i,0}) : i < \omega\}$ is consistent if and only if $\{\varphi(x;c_{0^{i} \frown 1}) : i < \omega\}$ is not $k$-inconsistent for any $k$. So for any $k$, we can find $i(0) < \ldots < i(k-1)$ so that $\{\varphi(x;c_{i(j)}) : j < k\}$ is consistent.  By $s$-indiscernibility we have $c_{0^{i(j)+1}} \equiv_{A (c_{0^{i(l)+1}} c_{0^{i(l)} \frown 1})_{l < j}} c_{0^{i(j)} \frown 1}$ for all $j < k$ and $\{\varphi(x;c_{0^{i(j)}}) : j < k\}$ is $m$-inconsistent for some $m$ (which does not depend on $k$), by our assumption that $\{\varphi(x;b_{i}) : i < \omega\}$ is inconsistent.  As $k$ is arbitrary, we obtain SOP$_{1}$ from Fact \ref{arrayequivalent} and compactness.  
\end{rem}
 
 \begin{cor}\label{nsop1_KL}
  Suppose $T$ satisfies existence. Then $T$ is NSOP$_1$ iff $T$ satisfies Kim's Lemma over arbitrary sets (i.e. the conclusion of Theorem \ref{KLforanyset}).
 \end{cor}
 \begin{proof}
  This follows by Theorem \ref{KLforanyset} and the fact that Kim's Lemma for Kim-dividing over models implies (and is equivalent to) NSOP$_1$ 
  (\cite[Theorem 3.15]{kaplan2017kim}).
 \end{proof}

%
 %
%
%
\section{Symmetry}\label{secsym}
Recall in this section we assume $T$ has NSOP$_1$ and satisfies existence for forking independence.

From the Kim's lemma for Kim-dividing we conclude:
\begin{prop}\label{kforkingequalskdividing}{(Kim-forking = Kim-dividing)}
For any $A$, if $\varphi(x;b)$ Kim-forks over $A$ then $\varphi(x;b)$ Kim-divides over $A$.
\end{prop}
\begin{proof}
Suppose $\varphi(x;b) \vdash \bigvee_{j< k} \psi_{j}(x;c^{j})$, where each
$\psi_{j}(x;c^{j})$ Kim-divides over $A$.    Let $(b_{i},c_{i}^{0}, \ldots, c^{k-1}_{i})_{i < \omega}$ be
a Morley sequence in $\tp(b,c^0,\dots,c^{k-1}/A)$. Since $(b_i)_i$ is a Morley sequence in 
$\tp(b/A)$, to get that $\varphi(x;b) $ Kim-divides over $A$ it is enough to show that 
$\{\varphi(x;b_{i}) : i < \omega \}$ is inconsistent. If not, then there is 
some $a\models \{\varphi(x;b_{i}) : i < \omega \}$. Then, by the pigeonhole principle,
we get that, for some $j<k$, $a$ realizes 
$\psi_{j}(x;c^{j}_i),$  for infinitely many $i$'s. As $(c_{i}^{j})_{i<\omega}$ is a Morley
sequence in $\tp(c^j/A)$, it follows from the Kim's lemma that
$\psi_{j}(x;c^{j}_i),$ does not Kim-divide over $A$, a contradiction.
\end{proof}

Below, we conclude that, under our assumptions, 
Kim-dividing (=Kim-forking) satisfies symmetry, by modifying
the notion of a  Morley tree and the proof of symmetry from \cite{kaplan2017kim}. As the argument is essentially the same, we only give a sketch.
We briefly recall the notation from \cite{kaplan2017kim}.

Recall the language $L_{s,\alpha}$ and its interpretation in trees were introduced at the beginning of Section \ref{section_Klemma}.
Our trees will be understood to be an $L_{s,\alpha}$-structure for some appropriate $\alpha$.  We recall the definition of a class of trees $\mathcal{T}_{\alpha}$ below:

\begin{defn}
Suppose $\alpha$ is an ordinal.  We define $\mathcal{T}_{\alpha}$ to be the set of functions $f$ such that 
\begin{itemize}
\item $\text{dom}(f)$ is an end-segment of $\alpha$ of the form $[\beta,\alpha)$ for $\beta$ equal to $0$ or a successor ordinal.  If $\alpha$ is a successor, we allow $\beta = \alpha$, i.e. $\text{dom}(f) = \emptyset$.
\item $\text{ran}(f) \subseteq \omega$.
\item finite support:  the set $\{\gamma \in \text{dom}(f) : f(\gamma) \neq 0\}$ is finite.    
\end{itemize}
We interpret $\mathcal{T}_{\alpha}$ as an $L_{s,\alpha}$-structure by defining 
\begin{itemize}
\item $f \unlhd g$ if and only if $f \subseteq g$.  Write $f \perp g$ if $\neg(f \unlhd g)$ and $\neg(g \unlhd f)$.  
\item $f \wedge g = f|_{[\beta, \alpha)} = g|_{[\beta, \alpha)}$ where $\beta = \text{min}\{ \gamma : f|_{[\gamma, \alpha)} =g|_{[\gamma, \alpha)}\}$, if non-empty (note that $\beta$ will not be a limit, by finite support). Define $f \wedge g$ to be the empty function if this set is empty (note that this cannot occur if $\alpha$ is a limit).  
\item $f <_{lex} g$ if and only if $f \vartriangleleft g$ or, $f \perp g$ with $\text{dom}(f \wedge g) = [\gamma +1,\alpha)$ and $f(\gamma) < g(\gamma)$
\item For all $\beta \leq \alpha$, $P_{\beta} = \{ f \in \mathcal{T}_{\alpha} : \text{dom}(f) = [\beta, \alpha)\}$.  
\end{itemize}
\end{defn}

\begin{defn}
Suppose $\alpha$ is an ordinal.  
\begin{enumerate}
\item (Restriction) If $w \subseteq \alpha$, the \emph{restriction of} $\mathcal{T}_{\alpha}$ \emph{to the set of levels  }$w$ is given by 
$$
\mathcal{T}_{\alpha} \upharpoonright w = \{\eta \in \mathcal{T}_{\alpha} : \min (\text{dom}(\eta)) \in w \text{ and }\beta \in \text{dom}(\eta) \setminus w \implies \eta(\beta) = 0\}.
$$
\item (Concatenation)  If $\eta \in \mathcal{T}_{\alpha}$, $\text{dom}(\eta) = [\beta+1,\alpha)$, and $i < \omega$, let $\eta \frown \langle i \rangle$ denote the function $\eta \cup \{(\beta,i)\}$.  We define $\langle i \rangle \frown \eta \in \mathcal{T}_{\alpha+1}$ to be $\eta \cup \{(\alpha,i)\}$.  We write $\langle i \rangle$ for $\emptyset \frown \langle i \rangle$.
\item (Canonical inclusions) If $\alpha < \beta$, we define the map $\iota_{\alpha \beta} : \mathcal{T}_{\alpha} \to \mathcal{T}_{\beta}$ by $\iota_{\alpha \beta}(f) = f \cup \{(\gamma, 0) : \gamma \in \beta \setminus \alpha\}$.
\item (The all $0$'s path) If $\beta < \alpha$, then $\zeta_{\beta}$ denotes the function with $\text{dom}(\zeta_{\beta}) = [\beta, \alpha)$ and $\zeta_{\beta}(\gamma) = 0$ for all $\gamma \in [\beta,\alpha)$.  This defines an element of $\mathcal{T}_{\alpha}$ if and only if $\beta \in \{\gamma \in\alpha\mid  \gamma \mbox{ is not limit}\}=:[\alpha]$.  
\end{enumerate}
\end{defn}

\begin{defn}\label{weaklyspread}
Suppose $(a_{\eta})_{\eta \in \mathcal{T}_{\alpha}}$ is a tree of tuples, and $C$ is a set of parameters.  
\begin{enumerate}
\item We say $(a_{\eta})_{\eta \in \mathcal{T}_{\alpha}}$ is \emph{weakly spread out over} $C$ 
if for all $\eta \in \mathcal{T}_{\alpha}$ with $\text{dom}(\eta) =[\beta+1,\alpha)$ for
some $\beta < \alpha$, 
$(a_{\unrhd \eta \frown \langle i \rangle})_{i < \omega}$ is a Morley sequence 
in $\tp(a_{\unrhd \eta \frown \langle 0 \rangle}/C)$.
\item Suppose $(a_{\eta})_{\eta \in \mathcal{T}_{\alpha}}$ is a tree which is weakly 
spread out and $s$-indiscernible over $C$ and for all $w,v \in [\alpha]^{<\omega}$ with $|w| = |v|$,
$$
(a_{\eta})_{\eta \in \mathcal{T}_{\alpha} \upharpoonright w} \equiv_{C} (a_{\eta})_{\eta \in \mathcal{T}_{\alpha} \upharpoonright v}
$$
then we say $(a_{\eta})_{\eta \in \mathcal{T}_{\alpha}}$ is a \emph{weakly Morley tree} over $C$.  
\item A \emph{weak tree Morley sequence} over $C$ is a $C$-indiscernible sequence of the form $(a_{\zeta_\beta})_{\beta \in [\alpha]}$ for some 
weakly Morley tree $(a_{\eta})_{\eta \in \mathcal{T}_{\alpha}}$ over $C$.  
\end{enumerate}
\end{defn}

\begin{prop}\label{chaincondition}
 If $a \ind^K_{A} b$ and  $I = (b_{i})_{i < \omega}$  is a Morley sequence in $\tp(b/A)$ with $b = b_{0}$, then there is $a' \equiv_{Ab} a$ such that $a'\ind^K_A I$ and $I$ is $Aa'$-indiscernible.  
\end{prop}

\begin{proof}
Extend $I$ to a sufficiently long Morley sequence $(b_{i})_{i < \kappa} $ in $\tp(b/A)$.  As $a \ind^K_{A} b$, we may assume that $b\equiv_{Aa} b_i$ for all $i<\kappa$ (by moving $a$ over $Ab$). Then, by Fact \ref{er}, find an $Aa$ indiscernible sequence $(b_j')_{j<\omega}$ based on 
$(b_{i})_{i < \kappa} $ over $Aa$. Let $a'$ to be the image of $a$ under an automorphism over $A$ sending $(b_j')_{j<\omega}$ to $I$. Then $I$ is $Aa'$ indiscernible. Now, it is enough to check that for all $n<\omega$, we have that $a' \ind^K_{A} b_{<n}$. But this follows from the 
indiscernibility and the Kim's lemma, as $(b_{kn}, b_{kn + 1}, \ldots, b_{kn+n-1})_{k < \omega}$ is a Morley sequence in $\tp(b_{<n}/A)$ (by left-transitivity of forking independence).
\end{proof}
Now, we modify the proof of Lemma 5.11 from \cite{kaplan2017kim} to get:

\begin{lem}\label{treeexistence}
If $a \ind^K_{A} b$, then for any ordinal $\alpha \geq 1$, 
there is a weakly spread out $s$-indiscernible tree
$(c_{\eta})_{\eta \in \mathcal{T}_{\alpha}}$ over $A$ such that if
$\eta \vartriangleleft \nu$ and $\text{dom}(\nu) = \alpha$, then 
$c_{\eta}c_{\nu} \equiv_{A} ab$.  
\end{lem}

\begin{proof}
We will argue by induction on $\alpha$.  
Suppose $\alpha=1$.
Assume $a \ind^K_{A} b$. Then, by Proposition \ref{chaincondition}, we can choose a Morley sequence in $\tp(b/A)$ which is $Aa$-indiscernible.
Put $c^{1}_{\emptyset} = a$ and $c^1_{\langle i \rangle} = b_{i}$.  Then
$(c^{1}_{\eta})_{\eta \in \mathcal{T}_{1}}$ satisfies the requirements.  

For the successor step, suppose for some $\alpha$ that we have constructed 
$(c^{\beta}_{\eta})_{\eta \in \mathcal{T}_{\beta}}$
for $1 \leq \beta \leq \alpha$ such that, if $\gamma < \beta \leq \alpha$ and
$\eta \in \mathcal{T}_{\gamma}$, then
$c^{\gamma}_{\eta} = c^{\beta}_{\iota_{\gamma \beta}(\eta)}$.  
Assume first that $\alpha$ is a successor ordinal.
By weak spread-outness, we know
that $(c^{\alpha}_{\unrhd {\langle i \rangle} })_{ i < \omega}$ is a
Morley sequence over $A$, which is, by $s$-indiscernibility over $A$,
$Ac^{\alpha}_{\emptyset}$-indiscernible. 
So $c^{\alpha}_{\emptyset} \ind^K_{A} (c^{\alpha}_{\unrhd \langle i \rangle})_{i < \omega}$. 
By extension for $\ind^K$,
we may find 
$c' \equiv_{A(c^{\alpha}_{\unrhd \langle i \rangle})_{i < \omega}} c^{\alpha}_{\emptyset}$ such
that 
$$
c' \ind^K_{A} (c^{\alpha}_{\eta})_{\eta \in \mathcal{T}_{\alpha}}.  
$$

Let $((c^{\alpha}_{\eta,i})_{\eta \in \mathcal{T}_{\alpha}} )_{ i < \omega}$ 
be a Morley sequence in 
$\tp((c^{\alpha}_{\eta})_{\eta \in \mathcal{T}_{\alpha}}/A)$
with $c^{\alpha}_{\eta,0} = c^{\alpha}_{\eta}$ for all $\eta \in 
\mathcal{T}_{\alpha}$.  
By Proposition  \ref{chaincondition}, we can find $c'' 
\equiv_{A(c^{\alpha}_{\eta})_{\eta \in \mathcal{T}_{\alpha}}} c'$ such that $c'' \ind^K_{A}
(c^{\alpha}_{\eta,i})_{\eta \in \mathcal{T}_{\alpha}, i < \omega}$ and
$((c^{\alpha}_{\eta,i})_{\eta \in \mathcal{T}_{\alpha}})_{i < \omega}$ is $Ac''$-indiscernible. 

Define a new tree $(d_{\eta})_{\eta \in \mathcal{T}_{\alpha+1}}$ by setting $d_{\emptyset} = c''$
and $d_{\eta\cup\{(\alpha,i)\}} = c^{\alpha}_{\eta,i}$ for all $\eta \in \mathcal{T}_{\alpha}$.  
Then let $(c^{\alpha+1}_{\eta})_{\eta \in \mathcal{T}_{\alpha+1}}$ be a tree $s$-indiscernible over
$A$ locally based on $(d_{\eta})_{\eta \in \mathcal{T}_{\alpha}}$.  By an automorphism, we may 
assume that $c^{\alpha+1}_{\iota_{\alpha\alpha+1}(\eta)} = c^{\alpha}_{\eta}$ for all $\eta \in 
\mathcal{T}_{\alpha}$.  This satisfies our requirements.  

If $\alpha$ is a limit ordinal, then, in order to repeat the above argument, find by compactness
$c$ which is $\ind^K$-independent
from $(c^{\alpha}_{\eta})_{\eta \in \mathcal{T}_{\alpha}} $ over $A$ such that 
for all $\nu\in\mathcal{T}_{\alpha}$ with $dom(\nu)=\alpha$,
$cc_{\nu} \equiv_{A} ab$
(notice that, since $\ind^K$-dependence is a property witnessed by formulas, one can express these
conditions by a type, and using the elements $c^{\beta}_{\emptyset}$, $\beta<\alpha$,
one gets that it is consistent).

Finally, for the limit step,  we obtain  
$(c^{\beta}_{\eta})_{\eta \in \mathcal{T}_{\delta}}$ 
from $(c^{\beta}_{\eta})_{\eta \in \mathcal{T}_{\beta}},$ $1 \leq \beta < \delta$  in the natural way
for any limit ordinal $\delta$.
\end{proof}
Now, using the same combinatorial arguments as in \cite{kaplan2017kim}, we can conclude:

\begin{lem}\label{movingtms}
Let $A$ be any set of parameters.
If $a \ind^K_{A} b$, then there is a weak tree Morley 
sequence $(a_{i})_{i < \omega}$ over $A$ which is $Ab$-indiscernible with $a_{0} = a$.  
\end{lem}

Notice that if $(a_{\eta})_{\eta \in \mathcal{T}_{\omega}}$ is a weak Morley tree over $A$, then 
if for  $\eta_{i} \in \mathcal{T}_{\omega}$ given by $\text{dom}(\eta_{i}) = [i,\omega)$ and
$$\eta_{i}(j) = \left\{ \begin{matrix}
1 & \text{ if } i = j \\
0 & \text{ otherwise},
\end{matrix} \right.
 $$
$(a_{\eta_i})_{i<\omega}$ is a Morley sequence over $A$.
Hence, repeating the proof of Proposition 5.13 from \cite{kaplan2017kim}, we get:
\begin{cor}{(Kim's lemma for weak tree Morley sequences)}\label{kimslemmafortms}
Let $A$ be any
set of parameters. The following are equivalent:
\begin{enumerate}
\item $\varphi(x;a)$ Kim-divides over $A$.
\item For some weak tree Morley sequence $(a_{i})_{i < \omega}$ over $A$ with $a_{0} = a$, $\{\varphi(x;a_{i}) : i < \omega\}$ is inconsistent.
\item For every weak tree Morley sequence $(a_{i})_{i < \omega}$ over $A$ with $a_{0} = a$, $\{\varphi(x;a_{i}) : i < \omega\}$ is inconsistent.  
\end{enumerate}
\end{cor}
\begin{cor}\label{symmetry}
$\ind ^K$ satisfies symmetry over arbitrary sets.
\end{cor}
\begin{proof}
Assume towards a contradiction that $a \ind^K_{A} b$ and $b \nind^K_{A} a$. 
By Lemma \ref{movingtms}, there is a weak tree Morley sequence $\langle a_i\mid i<\omega\rangle$ over $A$ with $a_{0} = a$ 
which is $Ab$-indiscernible.  Since $b \nind^K_{A} a$, there is some 
$\varphi(x;a) \in \text{tp}(b/Aa)$ which Kim-divides over $A$.  By Corollary \ref{kimslemmafortms}, 
$\{\varphi(x;a_{i}) : i < \omega\}$ is inconsistent.  
But $\models \varphi(b;a_{i})$ for all $i < \omega$ by indiscernibility, a contradiction.  
\end{proof} 

\section{3-amalgamation for Lascar strong types}\label{sec_symmetry}
In this section we will show that in any NSOP$_1$  theory satisfying the existence axiom, Lascar strong types have 3-amalgamation.
Throughout this section we assume that $T$ is NSOP$_1$ with existence.

Recall that $a\equiv^L_A b$ (equivalently $\Lstp(a/A)=\Lstp(b/A)$) if Lascar distance $d^L_A(a,b)$ is finite, i.e. if there  is a finite sequence 
$a=a_0,\dots, a_n=b$ such that $a_ia_{i+1}$ begins an $A$-indiscernible sequence for each $i<n$ (iff there  is a finite sequence 
$a=a_0,\dots, a_n=b$ and models $M_i\supseteq A$ such that 
$a_i\equiv_{M_i} a_{i+1}$ for $i<n$).  We say $a,b$ have the same KP-type over $A$  ($a\equiv^{KP}_Ab$) if $E(a,b)$ holds  for any $A$-type-definable bounded equivalence relation 
$E(x,y)$. It follows that $a\equiv^L_Ab$ implies $a\equiv^{KP}_Ab$. We say $T$ is {\em G-compact} if  KP-types are Lascar types, i.e. 
the converse holds for any $A$ and any $a,b$ of arbitrary arity. For more on Lascar types see for example \cite{kim2013simplicity}.

\begin{rem}\label{mdistance}
$a\equiv_A^Lb$ iff $Md_A(a,b)$ is finite, where $Md_A(a,b)$ is defined in the same way as Lascar distance of $a,b$ over $A$, except 
$A$-indiscernible sequences involved in the distance definition are all $A$-Morley sequences: It suffices to observe that for an 
  $A$-indiscernible $(a_0,a_1\in)I$,  there is $A$-Morley $J$ such that $a_0J\equiv_Aa_1J$. But this obviously follows from Fact \ref{seqamal}(1).
\end{rem}



\begin{rem}\label{basicfact}
By our assumptions and the results from previous sections, we have  symmetry: $a\ind^K_A b$ iff $b\ind_A^K a$; Kim-dividing=Kim-forking; and  extension: for any partial type $p(x)$ over $B$ not 
Kim-dividing over a set $A$, there is a completion $q(x)\in S(B)$ of $p$ not Kim-dividing over $A$. In addition, by Proposition \ref{chaincondition} (or by the same proof as for $\ind^K$ over a model in \cite{kaplan2017kim}), it follows that: if $a\ind_A^d bc$ and $b\ind_A^K c$, then $ab\ind_A^Kc$.\\
Moreover, in the same manner as in \cite[Lemma 5.9]{kaplan2017kim} we get that if $\la d_i\mid i<\omega\ra$ is weak tree Morley over $\emptyset$ with $d_i=a_ib_i$, then so are  $\la a_i\mid i<\omega\ra$ 
and $\la d_{n\cdot i} \dots d_{n\cdot i + n-1}\mid i<\omega \ra$ for each $n>0$.

\end{rem}

In the rest for convenience we take $\emptyset$ as the base set by naming the set;  when we say a type Kim-divides, we mean that it does so over $\emptyset$. 
Now, by a proof similar to that of the weak independence theorem in \cite{kaplan2017kim}, we obtain the following.

\begin{lem}\label{scissor}
Let $a\ind^K b$ and $a\ind^K c$. Then there is $e$ such that $ac\equiv ae$, $b\ind e$, and $a\ind^K be$. 
\end{lem}
\begin{proof}
We begin by establishing the following:

\medskip

\noindent\textbf{Claim.} There is $c'$ such that $ac'\equiv ac$ and $a\ind^K bc'$.

\medskip

\noindent\emph{Proof of claim.}  Due to symmetry, we have $c\ind^K a$ and  $b\ind^K a$. Hence there is
a Morley sequence $J=\langle a_i| i<\kappa\rangle$ with $a_0=a$  which is $b$-indiscernible, and there is $c''\equiv_a c$ such that $J$ is $c''$-indiscernible.  We can assume $\kappa$ is sufficiently large, so by Fact \ref{er}, 
there is $J'\equiv_{b} J_\omega =\la a_i\mid i<\omega\ra$  such that $J'$ is $bc''$-indiscernible.  Now if $f$ is some 
  $b$-automorphism sending $J'$ to $J_\omega$ then $J_\omega$ is $ac'$-indiscernible where $c'=f(c'')$. Hence 
  $bc'\ind^Ka$, so $a\ind^K bc'$. The claim  is proved.\qed

\medskip

Choose $c'$ as in the claim and take a Morley sequence $I=\langle b_ic_i| i \in \mathbb{Z} \rangle$ with $b_0c_0=bc'$. By Proposition \ref{chaincondition}, we can assume $I$ is $a$-indiscernible and $a\ind^K I$. Let $e=c_{-1}$. 
Then $a\ind^K be$, and since $I$ is a Morley sequence, we have $b \ind e$.  
\end{proof}

Now, we strengthen the lemma.

\begin{lem}\label{strongscissor}
Let $a\ind^K b$ and $a\ind^K c$. Then there is $e$ such that $ac\equiv^L ae$ and $b\ind e$ and $a\ind^K be$. 
\end{lem}
\begin{proof}
By extension and existence for $\ind$, there is a model $M$ such that $M\ind abc$. Then by Remark \ref{basicfact}, we have $Ma\ind^K b$ and 
$Ma\ind^K c.$ Now we apply  Lemma \ref{scissor} to $Ma, b,c$  to obtain $e$ such that $e\equiv_{Ma}c$, $b\ind e$,
and $a\ind^Kbe$. Then $ae\equiv^Lac$ follows.
\end{proof}

\begin{lem} \label{zigzag} (Zig-zag lemma) 
Let $b\ind^K c_0c_1$ and suppose there is a Morley sequence $I=\langle c_i: i<\omega \rangle$. Then  there is a weak tree Morley sequence $\langle b_ic_i |i<\omega\rangle$ such that 
for   all $i$, $b_ic_i\equiv bc_0$, and  for all  $j> i$, $b_ic_j\equiv bc_1$.
\end{lem}
\begin{proof}
We claim the following first.

\medskip

\noindent \textbf{Claim.} There is Morley  $J=\langle d_i| i<\omega\rangle\equiv I$ with $d_0=c_0$ such that
$b\ind^K J$, $bd_1\equiv bc_1$, and $d_{>0}$ is $bc_0$-indiscernible
 (so $bd_i\equiv bc_1$ for $i>0$).

\medskip 

\noindent\emph{Proof of claim.} Since $I_2=\langle c_{2i,2i+1}: i<\omega \rangle$ is  Morley as well, by Proposition \ref{chaincondition}
we can assume that $I_2$ is $b$-indiscernible, and $b\ind^K I_2$. Now we can assume the length of $I_2$ is a
sufficiently large $\kappa$, and consider $J'=\la  c_{2i+1}\mid i<\kappa\ra$. 
Then by Fact \ref{er},  there is $bc_0$-indiscernible $J_1=\la e_i\mid i<\omega\ra$ such that for each $n<\omega$, 
 $e_{\leq n}\equiv_{bc_0}   c_{2i_0+1}\dots c_{2i_n+1} $ 
for some $i_0<\dots <i_n<\kappa$.  Now put $J:=c_0J_1$. Then clearly  the claim is satisfied with this $J$.

\medskip

Now  since $d_{>0}$ is Morley and $e_0d_0$-indiscernible with $b=e_0$,  we see that $e_0d_0\ind^K d_{>0}$, and by extension 
  of $\ind^K$  there is $e_1$ such that 
$e_1d_{\geq 1}\equiv e_0J$, and $e_0d_0\ind^K e_1d_{\geq 1}$. Hence there is Morley $L_0=\la e^i_0d^i_0\mid  i<\omega\ra$ with $e^0_0d^0_0=e_0d_0$ 
such that 
$L_0$ is $e_1d_{\geq 1}$-indiscernible. Moreover again by Ramsey or Fact \ref{er}, there is $L_0e_1d_1$-indiscernible sequence $J_2 $ 
such that 
$d_0d_1J_2\equiv J$, so that $L_0e_1d_1\ind^K J_2$. Then by extension there is $e_2$ such that 
$e_2J_2\equiv e_1d_{\geq 1}(\equiv e_0J)$ and $L_0e_1d_1\ind^K e_2J_2$.  Hence there is Morley 
$L_{1}=\la L^i_0 e^i_{1}d^i_1\mid  i<\omega\ra$ with $L^0_0e^0_{1}d^0_1=L_0e_1d_1$ 
such that  $L_{1}$ is $e_2J_2$-indiscernible. Now let $d'_2$ be the first component of $J_2$. 
Then again by Fact \ref{er} and extension,  there are $e_3$ and $L_{1}e_2d'_2$-indiscernible $J_3 $ such that 
$d_0d_1d'_2J_3\equiv J$ and  $L_{1}e_2d'_2\ind^K e_3J_3(\equiv e_0J)$.  Then there is Morley 
$L_2=\la L^i_1e^i_{2}d^i_2\mid  i<\omega\ra$ with $L^0_1e^0_{2}d^0_2=L_1e_{2}d_2$, which is $e_3J_3$-indiscernible. 
  
\medskip

We sketch the rest of the proof. Notice that $L_0e_1d_1, L_1e_2d'_2, L_2e_3d'_3$ are naturally indexed by $\CT_1,\CT_2,\CT_3$, respectively.  
We iterate this argument for arbitrary large  $\alpha$ to get a  tree $L_\alpha e_{\alpha+1} d'_{\alpha+1}$ indexed by  $\CT_{\alpha+1}$. 
 Notice that each such tree is weakly spread out (see Definition \ref{weaklyspread}) by the way of construction.  
In the process we have kept the following conditions: 
Let $\la (uv)_{\eta_\beta}\mid \beta \leq \alpha+1\ra$ with $\dom(\eta_\beta)=[\beta, \alpha+1)$ be an arbitrary path  in the tree. Then 
 $\la v_{\eta_\beta}\mid \beta \leq \alpha+1\ra$ has the same EM-type as $J\equiv I$. In particular for any 
increasing sequence  $f(i)\leq \alpha+1$ with $i\in\omega$, we have 
$\la v_{\eta_{f(i)}}\mid  i<\omega\ra \equiv I$.  Moreover, for any $\beta$, we have $(uv)_{\eta_\beta}\equiv bc_0$, and for any $\gamma$ 
with $\beta<\gamma\leq \alpha+1$, we have  $u_{\eta_\beta}v_{\eta_\gamma}\equiv bc_1$.  

\medskip 

Consequently, when we shrink the tree into a weakly Morley tree (as in \cite[Section 5]{kaplan2017kim}) we can preserve the above conditions and the resulting 
weakly Morley tree  also meets 
the conditions.    Therefore we can find a weak tree Morley sequence described in this lemma.
\end{proof}

\begin{thm} \label{pre3-ap}
Let $b\ind^K c$, $a\equiv^L a'$, and $a\ind^K b$ with $p(x,b)=\mathrm{tp}(a/b)$, $a'\ind^K c$ with $q(x,c)=\mathrm{tp}(a'/c)$. 
Then $p(x,b)\cup q(x,c)$ does not Kim-divide.
\end{thm}
\begin{proof}
Since $a\equiv^L a'$, there is an automorphism $f$ fixing all the Lascar types over $\emptyset$ sending $a$ to $a'$.  Then clearly, $a\models p(x,b)\cup q(x,c'')$ 
where $c''=f(c)$, so $c\equiv^L c''$. In particular, $a\ind^K c''$. Due to Lemma 5.4, we can assume 
$b\ind c''$ and $a\ind^K bc''$, so $p(x,b)\cup q(x,c'')$ does not Kim-divide. 
Now applying 5.4 again to $b,c,c''$  we can find $c'$ such that $bc''\equiv bc'$ and $c''\equiv^Lc'(\equiv^L c)$ 
such that   $b\ind^K cc'$ and $c\ind c'$. Now as pointed out in Remark \ref{mdistance}, there are $c'=c_0,c_1,\ldots,c_n=c$ such that
each pair $c_ic_{i+1}$ starts a Morley sequence in $\tp(c_i)$. Moreover, due to extension of $\ind^K$ applied to $b\ind^K cc'$, we can assume $b\ind^K c_{\leq n}$.
Recall that $p(x,b)\cup q(x,c_0)$ does not Kim-divide \ (*), and we shall show that $p(x,b)\cup q(x,c_1)$ does not Kim-divide. (Then the same iterative argument shows that
each of  $p(x,b)\cup q(x,c_2),\ldots, p(x,b)\cup q(x,c_n)$ does not Kim-divide either, as wanted.)

Now due to Lemma \ref{zigzag}, there is a weak tree Morley sequence $I=\langle b'_ic'_i |i<\omega\rangle$ such that 
for any $i$, $b'_ic'_i\equiv bc_0$, and  for  $j> i$, $b'_ic'_j\equiv bc_1$. Then due to (*) and Corollary \ref{kimslemmafortms}, 
$$\bigcup_{i<\omega} p(x,b'_i)\cup q(x, c'_i) $$
is consistent. In particular, $\bigcup_{i<\omega} p(x,b'_{2i})\cup q(x, c'_{2i+1}) $
is consistent. Since $\langle b'_{2i}c'_{2i+1} |i<\omega\rangle$ is weak tree Morley as well (by Remark \ref{basicfact}), we have proved 
that $p(x,b)\cup q(x,c_1)$ does not Kim-divide.
\end{proof}

\begin{lem}\label{5.7} Assume $b\ind^Kc$. Then for any $a$ there is $e\equiv^L_ba$ such that
$eb\ind^K c$. Moreover, for any $d$ there is $b'\equiv^L_cb$ such that $b'\ind^K cd$. 
\end{lem}
\begin{proof} By extension of $\ind^K$, there is $M$ containing $b$ such that $M\ind^K c$. 
  Then by extension again there is $e\equiv_Ma$ (so $e\equiv^L_b a$), such that $eM\ind^Kc$, so $eb\ind^K c$. 
Similarly, there is a model $c\in N$ such that $b\ind^K N$. Then by extension there is $b'\equiv_N b$ (so $b'\equiv^L_c b$), 
such that $b'\ind^K Nd$, hence $b'\ind^K cd$. \end{proof}
  
\begin{thm}\label{3amalg} {\em ($3$-amalgamation of Lstp for Kim-dividing)} 
  {\em Let $b\ind^K c$, $a\equiv^L a'$, and $a\ind^K b$, $a'\ind^K c$. 
Then there is $a''$ such that $a''\ind^K bc$ and $a''\equiv^L_ba$, $a''\equiv^L_ca'$.} 
\end{thm}
\begin{proof}  Due to Lemma \ref{5.7}, there is $e\equiv^L_ba$ such that $eb\ind^Kc$, and there is $e'\equiv^L_ca'$ 
  such that $eb\ind^K ce'$.  Again by Lemma \ref{5.7}, there are  $a_0\equiv^L_b a$ and $a'_0\equiv^L_c a'$  such that $a_0\ind^K eb$ and  
  $a'_0\ind^K ce'$.  Note that  $ a'_0\equiv^L a'\equiv^L a\equiv^L a_0$. Hence by Theorem \ref{pre3-ap}, there is 
  $a''\models \tp(a_0/eb)\cup \tp(a'_0/ce')$ such that $a''\ind^K ebce'$. Moreover since 
  $a''\equiv_{eb} a_0$  and  $e\equiv^L_b a \equiv^L_b a_0$, it follows that $e\equiv_b^L a_0$ and $a''\equiv^L_b e\equiv^L_b a$. 
   Similarly we have $a''\equiv^L_c a'$ as wanted. \end{proof}
  
  Now by the same argument  using $3$-amalgamation as in simple theories (see \cite{kim2013simplicity}), 
   for any $a\equiv^L b$ with $a\ind^K b$ there is an indiscernible sequence starting with $a,b$. Hence 
  we conclude the following. 
  
\begin{cor} {\it  Assume $T$ is NSOP$_1$ with existence. Then $T$ is G-compact.  } 
\end{cor}

\section{The skew tree property}

Below, we introduce a combinatorial property STP (skew tree property), whose absence turns out to be equivalent to NSOP$_1$ under the assumption of 
existence.
%
  %
%
  %
  %
%

\begin{defn}
The formula $\varphi(x;y)$ has the \emph{skew tree property} if there are indiscernible sequences $\langle a_{i} : i < \omega + \omega \rangle$, $\langle b_{i} : i < \omega \rangle$ satisfying the following properties:
\begin{itemize}
\item $\{\varphi(x;a_{i}) : i < \omega + \omega\}$ is consistent
\item $\{\varphi(x;b_{i}) : i < \omega \}$ is inconsistent
\item $\langle b_{i} : i < \omega \rangle$ is $a_{<\omega}$-indiscernible
\item $\langle a_{\omega + i} : i < \omega \rangle$ is $a_{<\omega}b_{>0}$-indiscernible
\item $b_{0} = a_{\omega}$
\end{itemize}
The complete theory $T$ has the skew tree property if some formula does modulo $T$.
\end{defn}

\begin{lem} \label{between simple and NSOP1}
\begin{enumerate}
\item If $\varphi$ has SOP$_{1}$ then $\varphi$ has the skew tree property
\item If $\varphi$ has the skew tree property then $\varphi$ has the tree property.
\end{enumerate}
\end{lem}

\begin{proof}
(1)  Suppose $\varphi$ has SOP$_{1}$.  Then, by Fact \ref{arrayequivalent}(3) there is an array $(c_{i,0}, c_{i,1})_{i \in \omega+\omega+\omega}$ so that 
\begin{itemize}
\item $\{\varphi(x;c_{i,0}) : i <\omega+\omega+\omega\}$ is consistent
\item $\{\varphi(x;c_{i,1}) : i < \omega+\omega+\omega\}$ is inconsistent
\item The sequence $(\overline{c}_{i})_{i < \omega + \omega + \omega}$ is indiscernible and, for all $i \in \omega+\omega+\omega$, $c_{i,0} \equiv_{\overline{c}_{<i}} c_{i,1}$.
\end{itemize}
As $c_{\omega+\omega,0} \equiv_{\overline{c}_{<\omega+\omega}} c_{\omega+\omega,1}$, there is $\sigma \in \text{Aut}(\mathbb{M}/\overline{c}_{<\omega+\omega})$ so that $\sigma(c_{\omega+\omega, 0}) = c_{\omega+\omega,1}$.  Let $\langle a_{i} : i < \omega + \omega \rangle$ be defined by $a_{i} = c_{i,0}$ and $a_{\omega + i} = \sigma(c_{\omega + \omega +i,0})$ for $i < \omega$.  Let $\langle b_{i} : i < \omega+1\rangle$ be defined by $b_{i} = c_{\omega + i,1}$ for all $i < \omega + 1$.  In particular $b_{\omega} = c_{\omega+\omega,1}$.  We have $\langle a_{i} : i < \omega + \omega \rangle$ is an indiscernible sequence, since $\langle a_{i} : i < \omega + \omega \rangle = \sigma (\langle c_{i,0} : i < \omega \rangle \frown \langle c_{\omega + \omega + i,0} : i < \omega \rangle)$.  Similarly, as $\langle c_{\omega +\omega +i,0} : i < \omega \rangle$ is $\{c_{i,0},c_{\omega + i,1} : i < \omega\}$-indiscernible and $\sigma \in \text{Aut}(\mathbb{M}/\overline{c}_{<\omega+\omega})$, $\langle a_{\omega + i} : i < \omega \rangle$ is $a_{<\omega}b_{<\omega}$-indiscernible.  As $\langle c_{\omega + i,1} : i < \omega \rangle$ is $\{c_{i,0} : i < \omega\}$-indiscernible, $\langle b_{i} : i < \omega+1 \rangle$ is $a_{<\omega}$-indiscernible.  Finally, we know $\{\varphi(x;a_{i}) : i < \omega+\omega\}$ is consistent,  $\{\varphi(x;b_{i}) : i < \omega + 1\}$ is inconsistent, and $a_{\omega} = b_{\omega}$.  This clearly implies the skew tree property.  

(2)  Immediate from \cite[Lemma 2.3]{kim2001simplicity}.
\end{proof}

Define an independence relation $\ind^{*}$ by:  $a \ind^{*}_{A} b$ if there is a Morley sequence $\langle b_{i} : i < \omega \rangle$ over $A$ with $b_{0} = b$ which is $Aa$-indiscernible.  

\begin{prop} \label{amalgamation criterion}
Suppose that with respect to $T$, $\ind^{*}$ satisfies the independence theorem for Lascar strong types over arbitrary sets, that is, given any set $A$, $c_{0} \equiv^{L}_{A} c_{1}$, with $c_{0} \ind^{*}_{A} a$, $c_{1} \ind^{*}_{A} b$, and $a \ind^{*}_{A} b$, then there is $c_{*} \models \text{Lstp}(c_{0}/Aa) \cup \text{Lstp}(c_{1}/Ab)$ with $c_{*} \ind^{*}_{A} ab$.   Then $T$ does not have the skew tree property.  
\end{prop}

\begin{proof}
Suppose not, let $\varphi$, $\langle a_{i} : i < \omega + \omega \rangle$, $\langle b_{i} : i < \omega \rangle$ witness the skew tree property.  Let $A = \{a_{i} : i < \omega\}$.  From the definitions, we have $a_{\geq \omega}$ is an $Ab_{>0}$-indiscernible sequence.  As $\langle a_{i} : i < \omega + \omega\rangle$ is indiscernible, we have $a_{\omega + i} \ind^{u}_{A} a_{>\omega + i}$ for all $i$.  It follows that $a_{\geq \omega}$ is an $\ind^{u}$-Morley sequence enumerated reverse over $A$ (hence a Morley sequence over $A$).  As $a_{\omega} = b_{0}$ and $a_{\geq \omega}$ is $Ab_{>0}$-indiscernible, it follows that $b_{>0} \ind^{*}_{A} b_{0}$.  As $\langle b_{i} : i < \omega \rangle$ is an $A$-indiscernible sequence, it follows that $b_{>i} \ind^{*}_{A} b_{i}$ for all $i$.  Let $c \models \{\varphi(x;a_{i}) : i < \omega + \omega \}$ be a realization.  By Ramsey, compactness, and an automorphism, we may assume $\langle a_{i} : i < \omega + \omega \rangle$ is $c$-indiscernible, hence in particular $c \ind^{*}_{A} b_{0}$, since $b_{0} = a_{\omega}$.  

By assumption, $\{\varphi(x;b_{i}) : i \leq N\}$ is inconsistent for some $N$.  We will prove by induction on $k \leq N$ that there is some $c_{k} \models \{\varphi(x;b_{i}) : N-k \leq i \leq N\}$ with $c_{k} \ind^{*}_{A} (b_{i})_{N-k \leq i \leq N}$ to obtain a contradiction.  For $k = 0$, choose any $c_{0}$ with $c_{0}b_{N} \equiv_{A} cb_{0}$.  Now suppose for $k<N$, we have $c_{k} \models  \{\varphi(x;b_{i}) : N-k \leq i \leq N\}$ with $c_{k} \ind^{*}_{A} (b_{i})_{N-k \leq i \leq N}$.  We note that $b_{N-k-1}$ and $b_{N-k}$ start an $A$-indiscernible sequence so they have the same Lascar strong type over $A$.  Fix $\sigma \in \text{Autf}(\mathbb{M}/A)$ with $\sigma(b_{N-k}) = b_{N-k-1}$ and let $c' = \sigma(c_{k})$.  Then $c' \equiv^{L}_{A} c_{k}$, $\mathbb{M} \models \varphi(c';b_{N-k-1})$, and $c' \ind^{*}_{A} b_{N-k-1}$.  Recall that $(b_{i})_{N-k \leq i \leq N} \ind^{*}_{A} b_{N-k-1}$ so the independence theorem for Lascar strong types implies that there is $c_{*} \models \text{Lstp}(c'/Ab_{N-k-1}) \cup \text{Lstp}(c_{k}/A(b_{i})_{N-k \leq i \leq N})$ with $c_{*} \ind^{*}_{A} (b_{i})_{N-k-1 \leq i \leq N}$.  Note that $c_{*} \models \{\varphi(x;b_{i}) : N-k-1 \leq i \leq N\}$ so we may set $c_{k+1} = c_{*}$.  Continuing the induction, we find $c_{N} \models \{\varphi(x;b_{i}) : i \leq N\}$, contradicting the fact that this set of formulas is inconsistent.  This contradiction completes the proof.  
\end{proof}

  \begin{prop} Assume $T$ has existence (over any set). Then the following are equivalent.
\begin{enumerate}
\item $T$ is NSOP$_1$.
\item Kim's lemma holds for Kim-independence over any set.
\item The independence theorem for Lascar types for $\ind^{*}$ holds over any set.
\item The independence theorem for Lascar types for $\ind^{K}$ holds over any set.  
\item Kim-dividing satisfies symmetry over any set.
\item $T$ does not have STP.
\end{enumerate}
\end{prop}
\begin{proof}
By Corollary \ref{nsop1_KL} we know that (1) and (2) are equivalent. (1) implies (4) by Theorem \ref{3amalg}, 
(1) implies (5) by Corollary \ref{symmetry}, (4) implies (1) by \cite[Theorem 6.5]{kaplan2017kim}
and (5) implies (1) by Fact \ref{basic kimindep facts}.  Thus, (1),(2), (4), and (5) are equivalent.
Moreover (3) $\Rightarrow$ (6) $\Rightarrow$ (1) is Proposition \ref{amalgamation criterion} and Lemma \ref{between simple and NSOP1}.  Finally, assuming (1) (and hence also (2) and (4)), we have $\ind^{*} = \ind^{K}$ by (2) hence (3) follows from (4).  This completes the equivalence.  
\end{proof}

 \begin{cor}
There is a non-simple theory that does not have the skew tree property.
\end{cor}

\begin{proof}
All of the examples listed in Fact \ref{examples} are known to be NSOP$_{1}$ non-simple theories satisfying the existence axiom, and therefore do not have the skew tree property.  
\end{proof}

\begin{quest}\label{existence}
 Does the existence axiom hold in any NSOP$_1$ theory?
\end{quest}

 \begin{rem}\label{existence}
We observe that  if a formula $\varphi(x)$ (over $\emptyset$, say) implies $ \phi(x,a_0)\vee \psi(x,b_0)$ and each of
$\phi(x,a_0), \psi(x,b_0)$ $2$-divides over $\emptyset$ then $T$ has the strict order property: 
There are indiscernibles $\langle a_i\rangle,\langle b_i\rangle$ witnessing $2$-dividing of $\phi(x,a_0)$ and $\psi(x,b_0)$, respectively. 
Then for $i>0$, $\varphi(x)\wedge\phi(x,a_i)\models \varphi(x)\wedge\psi(x,b_0)$ and
$\varphi(x)\wedge\psi(x,b_i)\models \varphi(x)\wedge\phi(x,a_0)$. Now there is $a'_1$ such that $a'_1b_1\equiv a_1b_0$, so $\varphi(x)\wedge\phi(x,a'_1)$ implies $\varphi(x)\wedge\phi(x,a_0)$ but not the converse. Since $a'_1\equiv a_0$, $T$ has the strict order property. 
 Improving this idea to the general case seems difficult.
\end{rem}

Our proof of Kim's lemma relied on Lemma \ref{heirbase2}.  This lemma made heavy use of the assumption of non-forking existence for all types over all sets.  Consequently, it would be very interesting to know if the following local version could be proved without this assumption:

\begin{quest} \label{local version}
Suppose $T$ is NSOP$_{1}$, $p \in S(A)$, and $\langle a_{i} : i < \omega \rangle$ and $\langle b_{i} : i < \omega\rangle$ are both Morley sequences over $A$ in $p$.  Is it the case that $\{\varphi(x;a_{i}): i < \omega\}$ is consistent if and only if $\{\varphi(x;b_{i}) : i < \omega\}$ is consistent?
\end{quest}
 
 Finally, we ask several questions about the skew tree property:

\begin{quest}
\text{ }
\begin{enumerate}
\item Is NSOP$_{1}$ the same thing as $T$ does not have the skew tree property?
\item Is the property of not having the skew tree property preserved under reduct?
\item Is having the skew tree property equivalent to having the skew tree property witnessed by a configuration with $\{\varphi(x;b_{i}) : i < \omega\}$ 2-inconsistent?
\item Does not having the skew tree property imply every complete type has a global non-forking extension?
\end{enumerate}
\end{quest}

\bibliographystyle{plain}
\bibliography{ms.bib}{}

\end{document}